\documentclass[a4paper,11pt,reqno]{amsart}
\usepackage{amsmath,amssymb,amsfonts,amsthm}
\usepackage{latexsym,textcomp}
\usepackage{dsfont}
\usepackage[german, english]{babel}
\usepackage{hyperref}
\usepackage{enumerate}
\usepackage[utf8]{inputenc}

\newtheorem{lemma}{Lemma}[section]
\newtheorem{theorem}[lemma]{Theorem}

\newtheorem{proposition}[lemma]{Proposition}

\theoremstyle{definition}
\newtheorem{definition}[lemma]{Definition}
\newtheorem{example}[lemma]{Example}
\newtheorem*{remark}{Remark}

\numberwithin{equation}{section}
\newcommand{\comment}[1]{}

\newcommand{\Z}{{\mathbb Z}}
\newcommand{\R}{{\mathbb R}}

\newcommand{\N}{{\mathbb N}}

\newcommand{\rhoooo}{\varrho}

\newcommand{\av}[1]{\left\Vert #1\right\Vert}

\newcommand{\Hm}[1]{\leavevmode{\marginpar{\tiny%
$\hbox to 0mm{\hspace*{-0.5mm}$\leftarrow$\hss}%
\vcenter{\vrule depth 0.1mm height 0.1mm width \the\marginparwidth}%
\hbox to 0mm{\hss$\rightarrow$\hspace*{-0.5mm}}$\\\relax\raggedright
#1}}}

\begin{document}

\title{Note on intrinsic metrics on graphs}

\author{Daniel Lenz}
\address{Institut für Mathematik, Friedrich-Schiller-Universität Jena, 07743 Jena, Germany.} \email{daniel.lenz@uni-jena.de}

\author{Marcel Schmidt}
\address{Mathematisches Institut, Universität Leipzig, 04109 Leipzig, Germany.} \email{marcel.schmidt@math.uni-leipzig.de}

\author{Felix Seifert}
\address{Institut für Mathematik, Friedrich-Schiller-Universität Jena, 07743 Jena, Germany.} \email{felix.seifert@uni-jena.de}

\maketitle

\begin{abstract}
We study the set of intrinsic metrics on a given graph. This is a
convex compact set and it carries a natural order. We investigate
existence of largest elements with respect to this order. We show
that the only locally finite graphs which admit a largest intrinsic
metric are certain finite star graphs. In particular all infinite locally finite
graphs do not admit a largest intrinsic metric.  Moreover, we give a
characterization for the existence of intrinsic metrics with finite
balls for weakly spherically symmetric graphs.
\end{abstract}




\section{Introduction}
Analysis of the Laplace-Beltrami operator on a Riemannian manifold
is intimately linked to the geometry of the manifold. Of particular
importance is the metric geometry with respect to the geodesic
distance, see e.g. \cite{Gri99} for a survey. In his seminal works
\cite{Stu1,Stu2,Stu3} Sturm realized that many theorems which hold
on Riemannian manifolds can also be established for operators
induced by  strongly local Dirichlet forms when one replaces the
geodesic distance by so-called intrinsic metrics. On such Dirichlet
spaces there is the notion of a measure valued gradient, which we
denote by $\Gamma$. On a Riemannian manifold it is given by
$\Gamma(f) = |\nabla f|^2 {\rm dvol}$. The characterizing property
of an intrinsic metric $\rhoooo$ is a Rademacher type result: If $f$ is
$1$-Lipschtiz with respect to $\rhoooo$, then $\Gamma(f) \leq 1$ (this
notation means that the Radon-Nikdoym derivative with respect to the
underlying measure is   $\leq 1$). It turns out that there is a
largest intrinsic metric $\kappa$, which is given by
\begin{align*}
 \kappa (x,y) &= \sup\{ |f(x) - f(y)| \mid f \text{ continuous with } \Gamma(f) \leq 1\} \\
 &= \sup\{\rhoooo(x,y) \mid \rhoooo \text{ intrinsic}\}. 
\end{align*}
In the Riemannian setting this metric coincides with the geodesic
distance.

One assumption that is usually needed in the analysis of the
corresponding operators is compactness of balls with respect to an
intrinsic metric. Since large metrics lead to smaller balls and in
view of the Riemannian case, it  therefore makes sense to view
$\kappa$ as the canonical intrinsic metric of the strongly local
Dirichlet space. In fact, the metric $\kappa$ is exactly the metric
that the  mentioned works of Sturm rely on.

In recent years there was an increased interest in analysis of
unbounded graph Laplacians and more general non-local operators. It
was noted that the combinatorial graph distance is not well suited
for many analytical problems on graphs and, therefore, metrics
adapted to the situation were introduced simultaneously by  several
authors, see \cite{Fol11,Fol14} for graphs specifically, and
\cite{MU11} and \cite{GHM12} for some  more general non-local
operators and  \cite{FLW14} for a treatment of arbitrary  regular
Dirichlet forms.  The setting of \cite{FLW14} includes the
Laplace-Beltrami operator, operators of strongly local Dirichlet
forms, graph Laplacians and more general non-local operators.
Intrinsic metrics have since been successfully applied to tackle
various analytical problems, see e.g. the survey \cite{Kel15} and
the more recent book \cite{KLW21} for a comprehensive overview in
the discrete case.

It was noted already in \cite{FLW14} that in non-local situations
the maximum of two intrinsic metrics need not be intrinsic, which
led to the observation that in general there is no largest intrinsic
metric. This is a notable difference between  the strongly local
situation and the situation of non-local operators. Now, the
considerations  in  \cite{FLW14} provide just one example of a graph
for which no maximal intrinsic metric exists.  So, they naturally
raise the  question wether examples of graphs admitting maximal
intrinsic metrics exist at all or whether the non-existence of a
maximal intrinsic metric is a - so to speak - characteristic feature
of the non-locality.

In this paper we give a complete answer to this question for locally
finite graphs. Such a graph possesses  a largest intrinsic metric if
and only if it is a  star graphs satisfying some condition on the
measure, see Theorem~\ref{thm:star graphs}. So,  if one assumes
local finiteness, only for very few specific finite graphs a largest
intrinsic metric exists. In particular,  largest intrinsic metrics
can not exist  for infinite locally finite graphs. In this sense
lack of a largest  intrinsic metric is a characteristic feature of
non-locality in the infinite situation.

For graphs that are not locally finite we do not have a
comprehensive answer because in this case non-trivial intrinsic
metrics need not exist, see Example~\ref{example:no intrinsic
metric}. We discuss some partial results after
Proposition~\ref{proposition:kappa is not intrinsic}.

As mentioned above, for analysis of the graph Laplacian it is
desirable to have intrinsic metrics with relatively small balls,
which means large intrinsic metrics. Even if there is no largest
intrinsic metric, one can still ask for maximal intrinsic metrics
(in the sense of maximal elements with respect to the pointwise
order on intrinsic metrics). It turns out that maximal intrinsic
metrics always exist due to some compactness of the set of all
intrinsic metrics, see Proposition~\ref{prop:existence maximal
metrics},  but they are not unique if there is no largest intrinsic
metric (i.e. in the locally finite situation the graph is not a star
graph), see Theorem~\ref{theorem:uniqueness of maximal intrinsic
metrics}.

Summarizing Theorem~\ref{thm:star graphs} and
Theorem~\ref{theorem:uniqueness of maximal intrinsic metrics} one
can say that -  up to a few specific exceptions given by certain
finite graphs - locally finite graphs do not admit canonical
intrinsic metrics.

The related question whether a given graph possesses an intrinsic
metric with finite balls is hard to answer in general. There are
only two abstract criteria in this direction: For locally finite
graphs and path metrics there is a Hopf-Rinow type result, namely a
path metric has finite balls if and only if it is complete, see
\cite[Appendix~A]{HKMW13}. Also on locally finite graphs it was
recently noted that the existence of an intrinsic metric with finite
balls is related to the existence of a sequence of cut-off
functions, see \cite[Appendix~A]{LSW21}. The latter property is also
known as $\chi$-completeness and was characterized in \cite{BGJ19}
for radially symmetric trees. In Section~\ref{section:existence of
metrics with finite balls} we extend this to weakly spherically
symmetric graphs, which were introduced in \cite{KLW13}. We can
characterize when intrinsic metrics with finite balls and an
additional symmetry conditions exist. Under some growth condition on
the measure, this symmetry condition on the metric can be dropped,
see Theorem~\ref{theorem:weakly spherically symmetric}. In
particular, we obtain a sharp criterion for the existence of
intrinsic metrics with finite balls on antitrees with polynomially
growing spheres, see Example~\ref{exa:antitrees}.

Section~3 grew out of the  Bachelor's thesis of the third named
author.

\textbf{Acknowledgements.} D.~L.  and M.~S. gratefully acknowledge partial support of German Research Foundation (DFG).

\section{Set up and (non) existence of nontrivial intrinsic metrics}

Let $X \neq \emptyset$ be a countable set equipped with the discrete
topology. Hence, any function on $X$ is continuous, a subset of $X$
is compact if and only if it is finite and the $\sigma$-algebra
generated from the open sets consists just of all subsets. The space
of bounded real-valued functions on $X$ is denoted by
$\ell^\infty(X)$ and we write $C(X)$ for all real-valued functions
on $X$. Due to $X$ being discrete the support of $f \in C(X)$ is
given by ${\rm supp} f = \{x \in X \mid f(x) \neq 0\}$ and we write
$C_c(X)$ for all real valued functions of finite support. We write
$s\wedge t$ and $s\vee t$ for the minimum and maximum of the real
numbers $s$ and $t$ respectively. We  extend this to functions $f,g
\in C(X)$ pointwisely.

Let $m \colon X \to (0,\infty)$. We abuse notation and also view $m$
as a measure on all subsets of $X$ by
$$m(A) = \sum_{x \in A} m(x), \quad A \subseteq X. $$
%
%

A {\em weighted graph} on $X$ is a symmetric function $b \colon X \times X \to [0,\infty)$ such that $b(x,x) = 0$ for all $x \in X$ and such that the {\em weighted vertex degree} satisfies
$$\deg(x) := \sum_{y \in X} b(x,y) < \infty$$
for each $x \in X$.
%

Two points $x,y \in X$ are said to be {\em connected by an edge} if  $b(x,y) > 0$. In this case, we write $x \sim y$. The graph $b$ is called {\em locally finite} if for each $x \in X$ the set of its neighbors
$$  \{y \in X \mid x \sim y\}$$
is finite. In this case, we write $N(x) := |\{y \in X \mid x \sim y\}|$. A {\em path}  is a finite  sequence $\gamma = (x_0,\ldots,x_n)$ in $X$ with  $x_0 \sim x_1 \sim \ldots \sim x_n $. We say it {\em connects} $x$ and $y$ if they are contained in the path.  The graph $b$ is called {\em connected} if all $x,y \in X$ are connected by a path.

A {\em pseudo metric} on $X$ is a symmetric function $\sigma \colon X \times X \to [0,\infty)$ satisfying the triangle inequality. It is called {\em intrinsic} with respect to $b$ and $m$ if
$$\frac{1}{m(x)} \sum_{y \in X} b(x,y) \sigma(x,y)^2 \leq 1$$
for every $x \in  X$. We write $\mathfrak M = \mathfrak M_{b,m}$ for the set of all intrinsic pseudo metrics with respect to $b$ and $m$.

In applications the most important feature of intrinsic pseudo metrics is the following Rademacher type result: If $\sigma$ is an intrinsic pseudo metric and $f \colon X \to \R$ is $L$-Lipschitz with respect to $\sigma$ (i.e. $|f(x) - f(y)| \leq L \sigma(x,y)$ for all  $x,y \in X$), then
$$|\nabla f| := \left( \frac{1}{m(x)} \sum_{y \in X} b(x,y) |f(x) -f(y)|^2\right)^{1/2} \leq L.$$

 Clearly, the trivial pseudo metric $\sigma = 0$ is always an intrinsic pseudo metric and so the first question is whether non-trivial intrinsic metrics always exist. The following example provides a positive answer for locally finite (connected) graphs.

 \begin{example}[Path pseudo metrics] \label{example:path metrics}
 Assume the graph $b$ is connected and let $w \colon X \times X \to [0,\infty)$. To a path $\gamma = (x_0,\ldots,x_n)$ we associate its {\em length} (with respect to $w$) by
 $$L_w(\gamma) = \sum_{k = 1}^n w(x_{k-1},x_k)\wedge w(x_{k},x_{k-1}).$$
 The function
 $$d_w \colon X \times X \to [0,\infty), \, d_w(x,y) = \inf \{L_w(\gamma) \mid \gamma \text{ path connecting }x \text{ and }y\}$$
 is called the {\em path pseudo metric} induced by the weight  $w$.
It is readily verified that $d_w$ is a pseudo metric. Note that
$d_w$ only depends on the values  $w(x,y)$ with $x \sim y$. Hence,
in some examples below we only specify the values of such $w$ on
neighbors.  Note also that we can assume without loss of generality
that $w$ is symmetric (i.e. satisfies $w(x,y) = w(y,x)$ for all
$x,y\in X$ with $x\sim y$) as we could otherwise replace it by the
symmetric $w(x,y) \wedge w(y,x)$.

Path pseudo metrics have the following characteristic feature.

\begin{proposition}[Characterisation of path pseudo metrics]\label{char-path}  Let $(X,b)$ be a connected graph. Let $\varrho$ be a
pseudo metric on $X$.

(a) If  $\varrho$ is a path pseudo metric induced by the  weight $w$,
then any pseudometric $\sigma$ with $\sigma(x,y)\leq w(x,y)$ for all
$x,y\in X$ with $x\sim y$ satisfies $\sigma \leq \varrho$.

(b) The pseudometric $\varrho$ is a path pseudo metric if and only if
any pseudometric $\sigma$ with $\sigma(x,y)\leq \varrho(x,y)$ for all
$x,y\in X$ with $x\sim y$ satisfies $\sigma \leq \varrho$.

\end{proposition}
\begin{proof}
(a) Assume without loss of generality that $w$ is symmetric. Let
$x,y\in X$ be given and $(x_0,\ldots, x_n)$ a path connecting $x$
and $y$. Then, the triangle inequality (for $\sigma$)  gives
$$\sigma(x,y)\leq \sum_{k=1}^n \sigma(x_{k-1},x_k) \leq \sum_{k=1}^n
w(x_{k-1},x_k).$$ Taking the infimum over all path connecting $x$
and $y$ we find $\sigma \leq \varrho$.

(b) From (a) we easily infer the 'only if' part. To show the 'if'
part assume that $\varrho$ has the feature that  any pseudometric
$\sigma$ with $\sigma(x,y)\leq\varrho(x,y)$ for all $x,y\in X$ with
$x\sim y$ satisfies $\sigma \leq \varrho$. Consider now the pseudo
path metric $\sigma:=d_\varrho$. Then, clearly $\sigma(x,y)\leq
\varrho (x,y)$ holds for all $x,y\in X$ with $x\sim y$. So, we infer
$$\sigma \leq \varrho.$$ Moreover, as $\varrho$ is a pseudo metric we
can apply the triangle inequality to $\varrho$ to find for any path
$(x_0,\ldots, x_n)$ connecting $x$ and $y$ that
$$\varrho(x,y)\leq \sum_{k=1}^n \varrho (x_{k-1},x_k)$$
holds. Taking the infimum over all path we obtain
$$\varrho \leq d_\varrho = \sigma.$$
Putting this together we find that  $\varrho = d_\varrho$ is a path
metric.
\end{proof}

If the graph is locally finite and $w(x,y) > 0$ for all $x \sim y$,
then $d_w$ is even  a metric inducing the discrete topology on $X$.
For $w = 1$ the metric $d:= d_1$ is called {\em combinatorial
distance}, as it counts the least number of edges in a path
connecting $x,y$.

The path pseudo metric $d_w$ satisfies $d_w(x,y) \leq w(x,y) \wedge
w(y,x)$ for $x \sim y$. Hence, it is intrinsic if
 $$\frac{1}{m(x)} \sum_{y \in X} b(x,y) w(x,y)^2 \leq 1.$$
 One function $w$ satisfying this inequality is given by
 $$w(x,y) = \frac{m(x)^{1/2}}{\deg(x)^{1/2}} \wedge  \frac{m(y)^{1/2}}{\deg(y)^{1/2}}.  $$
This provides a concrete example of an intrinsic pseudo metric that
is a path pseudo  metric.

The combinatorial distance $d$ satisfies $d(x,y) = 1$ for $x \sim
y$. Hence, it is intrinsic if and only if $ \deg/m$ is bounded by
$1$.
 \end{example}

Another class of intrinsic pseudo metrics is induced by functions on
$X$. This is discussed next.

\begin{example}[Intrinsic metrics induced by functions]\label{example:intrisic metrics from functions}
 Let $f \colon X \to \mathbb R$. Then
 $$\sigma_f \colon X \times X \to [0,\infty), \quad \sigma_f(x,y) = |f(x) - f(y)|$$
 is obviously a pseudo metric on $X$. It is intrinsic if and only if $|\nabla f| \leq 1$. Moreover, it is a metric if and only if $f$ is injective.
\end{example}

The following lemma introduces a metric which dominates all
intrinsic pseudo metrics. It leads to examples of graphs that do not
possesses any nontrivial intrinsic pseudo metric and to compactness
of the set of all intrinsic pseudo metrics $\mathfrak M$ with
respect to pointwise convergence.

\begin{lemma}\label{lemma:universal bound}
Let $(X,b)$ be a graph and consider $s \colon X \times X \to [0,\infty)$ with
$$s(x,y) = \frac{m(y)^{1/2}}{b(x,y)^{1/2}}$$
for $x \sim y.$ Then any intrinsic pseudo metric $\sigma$ satisfies
$\sigma \leq d_s$.  In particular, $\sigma(x,y) \leq s(x,y)$ for $x
\sim y$.
\end{lemma}

\begin{proof} Let $\sigma$ be intrinsic.  For $z \sim x$ the inequality $\sum_{y
\in X} b(x,y) \sigma(x,y)^2 \leq m(x)$  implies
$$\sigma(x,z) \leq \frac{m(x)^{1/2}}{b(x,z)^{1/2}}= s(z,x)$$
where we use the symmetry of $b$ to obtain the last equality. As
$\sigma$ is a metric this gives
$$\sigma (x,y) = \sigma (y,x) \leq s(x,y)$$
whenever $x\sim y$ holds.  From (a) of Proposition~\ref{char-path} we then infer $\sigma\leq d_s$.
\end{proof}

\begin{example}[Graphs without nontrivial intrinsic pseudo metric] \label{example:no intrinsic metric}
Assume
\begin{equation*}
\inf_{z \in X} \left(\frac{m(z)^{1/2}}{b(x,z)^{1/2}}+\frac{m(z)^{1/2}}{b(y,z)^{1/2}} \right)=0
\end{equation*}
for all $x,y \in X$ (here we use the convention $m(z)^{1/2}/b(x,z)^{1/2} = \infty$ if $x \not \sim z$). Then the triangle inequality and the previous lemma imply that there is no nontrivial intrinsic pseudo metric with respect to $b$ and $m$. In this case, Example~\ref{example:intrisic metrics from functions} shows that any $f \in C(X)$ with $|\nabla f| \in \ell^\infty$ is constant, which is a quite remarkable property.

We now explicitly construct such a graph. We let $X= \N$,
\begin{equation*}
b \colon \N \times \N \to [0,\infty), \quad b(i,j) = \begin{cases}
          \frac{1}{i^2+j^2} &\text{if } i \neq j\\
          0 &\text{if } i = j
         \end{cases}
\end{equation*}
 and
\begin{equation*}
m \colon \N \to (0,\infty), \quad m(i)=\frac{1}{i^3}.
\end{equation*}
It is clear that $b$ is symmetric and satisfies the summability condition. Moreover,
\begin{equation*}
\frac{m(k)^{1/2}}{b(i,k)^{1/2}}+ \frac{m(k)^{1/2}}{b(j,k)^{1/2}}=\frac{(i^2+k^2)^{1/2}}{k^{3/2}}+ \frac{(j^2+k^2)^{1/2}}{k^{3/2}} \rightarrow 0, \quad k\rightarrow \infty,
\end{equation*}
which shows that $(\N,b)$ has no nontrivial intrinsic metric with respect to $m$.
\end{example}

\begin{remark}
Non-existence of nontrivial intrinsic pseudo metrics is not a special feature of graphs or non-local operators. Indeed, for certain strongly local Dirichlet forms on fractals the measures $\Gamma(f)$ mentioned in the introduction are always singular to the underlying measure and hence the only intrinsic pseudo metric is the trivial one, see e.g. \cite{Hin05}.  
\end{remark}

\begin{proposition} \label{proposition:compactness}
Let $b$ be a graph over $(X,m)$. Then $\mathfrak M$ is a compact
convex subset of $C(X \times X)$ equipped with the topology of
pointwise convergence.
\end{proposition}
\begin{proof}
The convexity  of $\mathfrak M$ follows directly from the the
definition of intrinsic pseudo metrics and the convexity of the
function $[0,\infty) \to [0,\infty), \, x \mapsto x^2$.

Since $X \times X$ is at most countable, the topology of pointwise
convergence on $C(X \times X)$ is metrizable. Let $d_s$ be the
metric introduced in Lemma~\ref{lemma:universal bound}.  By a
standard diagonal sequence argument the set
$$\{f \in C(X \times X) \mid 0 \leq f \leq d_s\} $$
is compact in $C(X \times X)$ with respect to pointwise convergence. By Lemma~\ref{lemma:universal bound} it contains $\mathfrak M$ and so it suffices to show that $\mathfrak M$ is closed.  Let $(\sigma_n)$ in $\mathfrak M$ and $\sigma \in C(X \times X)$ with $\sigma_n \to  \sigma$ pointwise. The pointwise limit of pseudo metrics is clearly a pseudo metric and Fatou's lemma yields
$$\sum_{y \in X} b(x,y) \sigma(x,y)^2 \leq \liminf_{n \to \infty} \sum_{y \in X} b(x,y) \sigma_n(x,y)^2 \leq m(x) $$
for all $x \in X$. Hence, $\sigma  \in \mathfrak M$.
\end{proof}

Clearly, the set $C(X\times X)$ carries a natural order $\leq$ with
$F\leq G$ if $F(x,y)\leq G(x,y)$ for all $x,y\in X$. This order
induces an order on $\mathfrak M$ and we look for largest elements
with respect to this order. This is the topic of the next section.

\section{Maximal intrinsic metrics}
As mentioned in the introduction intrinsic pseudo metrics with small
balls are useful in applications. Small balls correspond to large
pseudo metrics. Hence,  we study largest  intrinsic pseudo metrics
in this section. Here, largest  refers to the natural order
structure on the set of intrinsic metrics induced by pointwise
comparison (compare discussion at the end of previous section).

One candidate for a largest  intrinsic pseudo metric is the
pointwise sup of all intrinsic metrics. For connected locally finite
graphs it turns out that only on special star graphs this leads to
an intrinsic metric (Theorem \ref{thm:star graphs}).

In the quest for largest intrinsic metrics one may also turn  to
maximal intrinsic metrics (i.e. intrinsic metrics $\sigma$ which
agree with any intrinsic metric $\varrho$ with $\sigma \leq
\varrho$). Such maximal intrinsic metrics can be shown to exist on
any graph with the help of Zorn's lemma (Proposition
\ref{prop:existence maximal metrics}). However, except in the case
of the special star graphs already mentioned, there will exist more
than one such maximal intrinsic metric (Theorem
\ref{theorem:uniqueness of maximal intrinsic metrics}). So, there is
a lack of uniqueness.

Let $b$ be a graph over $(X,m)$. From  Lemma~\ref{lemma:universal bound}  we infer that the function
$$\kappa = \kappa_{b,m}  = \sup \{ \sigma \mid \sigma \in \mathfrak M_{b,m}\}$$
is finite with $\kappa \leq d_s$. As a pointwise supremum of pseudo metrics it is a pseudo metric itself. We call it the {\em canonical pseudo metric} of $b$ over $(X,m)$.  If the graph is locally finite, there are intrinsic metrics (cf. Example~\ref{example:path metrics}) and so $\kappa$ is indeed a metric in this case.

\begin{proposition}
Let $b$ be a graph over $(X,m)$. For $x,y \in X$ we have
$$\kappa(x,y) = \sup\{|f(x) - f(y)| \mid f \in C(X) \text{ with }´ |\nabla f| \leq 1\}.$$
\end{proposition}
\begin{proof}
Let $S_{xy} = \sup\{|f(x) - f(y)| \mid  |\nabla f| \leq 1\}$.  As noted in Example~\ref{example:intrisic metrics from functions} for $f \in C(X)$ with $|\nabla f| \leq 1$ the pseudo metric $\sigma_f$ is intrinsic. Since $\sigma_f(x,y) = |f(x) - f(y)|$, this leads to $\kappa(x,y) \geq S_{xy}$.

Now let $\sigma$ be an intrinsic pseudo metric. For $w \in X$ the function $f_w = \sigma(\cdot,w)$ is $1$-Lipschitz and therefore satisfies $|\nabla f_w| \leq 1$. For $x,y \in X$ we obtain $\sigma(x,y) = \sup \{|f_w(x) - f_w(y)| \mid w \in X\}$. Combining both observations leads to $S_{xy} \geq \sigma(x,y)$. Since $\sigma$ was arbitrary, this leads to $\kappa(x,y) \leq S_{xy}$.
\end{proof}

\begin{definition}[Star graph]
  We say a graph $b$ over $X$ is a {\em star graph} if it is connected and there exists $p \in X$ such that $N(x) = 1$ for all $x \in X \setminus\{p\}$. In this case, $p$ is called {\em center} of $b$.

A graph $b$ over $X$ is called a {\em galaxy} if all of its connected components are star graphs.

\end{definition}
\begin{remark}
 By definition star graphs are connected. Hence, if $p$ is a center, then every vertex from $X \setminus \{p\}$ is connected to $p$. The center of  a star graph is unique if $|X| \geq 3$.
%
\end{remark}

\begin{proposition}\label{proposition: star graphs}
 Let $b$ be a star graph over $(X,m)$. The following assertions are equivalent.
 \begin{enumerate}[(i)]
  \item The canonical pseudo metric $\kappa$ is intrinsic.
  \item There exists a center $p \in X$ of $b$ such that
  $$\sum_{x \in X \setminus \{p\}} m(x) \leq m(p).$$
  \end{enumerate}
\end{proposition}
\begin{proof}
Let $p \in X$ be a center of $b$. If $X = \{p\}$, then there is nothing to show and we therefore assume $|X| \geq 2$.  We first prove
\begin{equation*}
\kappa(x,p)^2 = \frac{m(x) \wedge m(p)}{b(x,p)}
\end{equation*}
for $x \in X \setminus \{p\}$. We already know from Lemma~\ref{lemma:universal bound} that the left hand side is smaller than the right hand side.

For the opposite inequality for $x \in X \setminus \{p\}$ we consider the function
$${f= \left( \frac{(m(x) \wedge m(p))^{\frac{1}{2}}}{b(x,p)^{\frac{1}{2}}} \right) 1_{\{x\}}}.$$
Using that $b$ is a star graph with center $p$, it is readily verified that $|\nabla f| \leq 1$. Hence, the discussion in Example~\ref{example:intrisic metrics from functions} shows that $\sigma_f$ is an intrinsic pseudo metric. Since
$$\sigma_f(x,p)^2 = (f(x) - f(p))^2 =  \frac{m(x) \wedge m(p)}{b(x,p)} $$
and since $\kappa$ is the supremum over all intrinsic pseudo metrics, we obtain the desired equality.

If $x \in X \setminus \{p\}$, our formula for $\kappa$ shows
$$\frac{1}{m(x)} \sum_{y \in X} b(x,y) \kappa(x,y)^2 = \frac{b(x,p)}{m(x)}  \frac{m(x) \wedge m(p)}{b(x,p)} \leq 1.  $$
In particular, $\kappa$ being intrinsic can only fail at $p$. But at $p$ we have
$$\frac{1}{m(p)} \sum_{y \in X} b(p,y) \kappa(p,y)^2 =\frac{1}{m(p)} \sum_{y \in X \setminus\{p\}} (m(y) \wedge m(p)).  $$
It is readily verified that this sum is $\leq 1$ if and only if $\sum_{y \in X \setminus\{p\}} m(y) \leq m(p)$. This shows the equivalence of (i) and (ii).
\end{proof}

\begin{proposition}\label{proposition:kappa is not intrinsic}
 Let $b$ be a locally finite connected graph over $(X,m)$ and assume it is not a star graph. Then $\kappa$ is not intrinsic.
\end{proposition}
\begin{proof}
 We assume that $\kappa$ is intrinsic and construct another intrinsic pseudo metric which is larger than $\kappa$ in one argument. Since $b$ is connected but no star graph, there exist $x,y \in X$ with $x \sim y$ and $N(x),N(y) \geq 2$.
Using that $b$ is locally finite we can assume without loss of
generality  that  $y$  satisfies
 $$\kappa(x,y) = \min \{ \kappa(x,z) \mid z \sim x \text{ and } N(z) \geq 2\}.$$

Moreover, local finiteness implies that $\kappa$ is a metric and we
obtain
$$s:= \min (\{\kappa(x,z) \mid z \sim x \} \cup \{\kappa(y,z) \mid z \sim y \}) > 0.    $$
 %
%
%
%
By our choice of $s$ and since $\kappa$ is intrinsic, there exists
$0 <\varepsilon < s/3$ with

\begin{equation*}
\frac{1}{m(x)} \left(\sum\limits_{z \in X\setminus \{y\} }b(x,z) (k(x,z)-\frac{s}{3})^2 +b(x,y) (k(x,y) + \varepsilon)^2 \right) \leq  1
\end{equation*}
and
\begin{equation*}
\frac{1}{m(y)} \left(\sum\limits_{z \in X\setminus \{x\} }b(y,z) (k(y,z)-\frac{s}{3})^2 +b(y,x) (k(y,x) + \varepsilon)^2 \right) \leq  1.
\end{equation*}
%
%
Next we consider the weight
$$w = \kappa + \varepsilon 1_C  - \frac{s}{3} 1_D, $$
with $C = \{x,y\} \times \{x,y\}$ and $D = \{x,y\} \times (X \setminus \{x,y\}) \cup (X \setminus \{x,y\}) \times \{x,y\}$. The above inequalities and $\kappa$ being intrinsic yield
$$\frac{1}{m(u)}\sum_{z \in X} b(u,z)w(u,z)^2 \leq 1$$
for all $u \in X$. Hence,  the induced path metric $d_w$ is
intrinsic.

We finish the proof by showing $d_w(x,y) \geq \kappa(x,y) +
\varepsilon$. Let $\gamma = (x_0,\ldots,x_n)$ be a path connecting
$x$ and $y$. Assume without loss of generality that the $x_j$,
$j=1,\ldots, n$,  are pairwise different. We prove $L_w(\gamma) \geq
\kappa(x,y) + \varepsilon$.

If $n = 1$, then $x_0 = x, x_1 = y$ and we obtain
$$L_w(\gamma) = w(x,y) = \kappa(x,y) + \varepsilon. $$

If $n\geq 2$, without loss of generality we can assume $N(x_i) \geq
2$ for all $i = 1,\ldots,n-1$ (otherwise $(x_0,\ldots
x_{i-1},x_{i+1},\ldots,x_n)$ is a shorter path connecting $x$ and
$y$). In particular, we then have $N(x_1)\geq 2$.

Now, by the very definition of path we have $x_1 \neq x, x_{n-1}
\neq y$ and - as the elements of the path are pairwise different -
we also have   $x_1\neq y$ and $x_{n-1} \neq x$. So, $(x,x_1)$ and
$(x_{n-1},y)$ both belong to $D$. Hence, we obtain
\begin{align*}
 L_w(\gamma) \geq w(x,x_1) + w(x_{n-1},y) = \kappa(x,x_1) - \frac{s}{3} + \kappa(x_{n-1},y) - \frac{s}{3}.
\end{align*}
The definition of $s$ yields $ \kappa(x_{n-1},y)  \geq s$. Our
choice of $y$ (together with $N(x_1) \geq 2$)  yields $\kappa(x,x_1)
\geq \kappa(x,y)$. Altogether we arrive at
$$L_w(\gamma) \geq \kappa(x,y) + \frac{s}{3} \geq \kappa(x,y) + \varepsilon,$$
where we used $\varepsilon < s/3$ in the last estimate.
\end{proof}
\begin{remark}
Example~\ref{example:no intrinsic metric} shows that the previous proposition does not hold in general on graphs that are not locally finite. The graph without nontrivial intrinsic pseudo metric constructed there is a complete graph (every vertex is connected to every other vertex). Note however that the previous proposition holds true (with essentially the same proof) under weaker conditions. One only needs the existence of $x \in X$  with  $2 \leq N(x) < \infty$ and  $N(z) < \infty$ for all $z \sim x$ and the existence of $y \sim x$ with $N(y) \geq 2$.   In other words, the combinatorial $2$-ball around $x$ is locally finite and the graph is not a star graph.  
\end{remark}

\begin{theorem}\label{thm:star graphs}
 Let $b$ be a connected locally finite graph over $(X,m)$. The following assertions are equivalent.

\begin{enumerate}[(i)]
 \item The canonical metric $\kappa$ is intrinsic.
 \item $b$ is a star graph with center $p \in X$ such that
 $$\sum_{x \in X \setminus \{p\}} m(x) \leq m(p).$$
\end{enumerate}
\end{theorem}
\begin{proof}
(i) $\Rightarrow$ (ii):  It follows from  Proposition~\ref{proposition:kappa is not intrinsic} that $\kappa$ is not intrinsic if $b$ is not a star graph.  The criterion on $m$ follows Proposition~\ref{proposition: star graphs}.

(ii) $\Rightarrow$ (i): This follows from Proposition~\ref{proposition: star graphs}.
\end{proof}

\begin{remark}
 We formulated the theorem for connected graphs but it also holds on locally finite non-connected graphs. In this case, $\kappa$ is an intrinsic metric if and only if $b$ is a galaxy and the inequality in (ii) holds on each of its stars (i.e. every connected component).
\end{remark}

Next we discuss how $\kappa$ being intrinsic is related to uniqueness of maximal intrinsic pseudo metrics. To this end, we need a finitary description of $\kappa$ being intrinsic as given by the next proposition.

\begin{proposition}\label{lemma:finite sups are sufficient}
Let $b$ be a graph over $(X,m)$. The following assertions are equivalent.
\begin{enumerate}[(i)]
 \item $\kappa$ is an intrinsic metric.
 \item For all intrinsic metrics $\rhoooo,\sigma$ the metric $\rhoooo \vee \sigma$ is intrinsic.
\end{enumerate}
\end{proposition}
\begin{proof}
(i) $\Rightarrow$ (ii): If $\rhoooo,\sigma$ are intrinsic, then $\rhoooo,\sigma \leq \kappa$, which implies $\rhoooo \vee \sigma \leq \kappa$. Since $\kappa$ is intrinsic, this implies that $\rhoooo \vee \sigma $ is intrinsic.

(ii) $\Rightarrow$ (i): Since $X \times X$ is countable, we find a sequence $(\sigma_n)$ in $\mathfrak M$ with
$$\kappa  = \sup \{\sigma_n \mid  n \in \N\} = \lim_{N \to \infty} \sup \{\sigma_n \mid  1 \leq n \leq N\}, $$
where the convergence holds pointwise. Iterating the assumption yields $\sup \{\sigma_n \mid  1 \leq n \leq N\} \in \mathfrak M$. Since $\mathfrak M$ is closed with respect to pointwise convergence, we arrive at (i).
\end{proof}

Even though $\kappa$ may not be an intrinsic metric there always exist maximal intrinsic pseudo metrics. More precisely, an intrinsic pseudo metric $\sigma$ is called {\em maximal} if for any other intrinsic pseudo metric $\rhoooo$ the inequality $\sigma(x,y) \leq \rhoooo(x,y)$ for all $x,y \in X$ implies $\rhoooo = \sigma$.

\begin{proposition}[Existence of maximal intrinsic metrics]\label{prop:existence maximal metrics}
Let $b$ be a graph over $(X,m)$. For every intrinsic pseudo metric $\sigma$ there exists a maximal intrinsic pseudo metric $\rhoooo$ with $\sigma \leq \rhoooo$.
\end{proposition}
\begin{proof}
We use Zorn's lemma. The pointwise order on
$$\mathfrak M_\sigma := \{\rhoooo \mid \rhoooo  \in  \mathfrak M  \text{ with } \sigma \leq \rhoooo\}$$
is a partial order. Let $\mathfrak F \subseteq \mathfrak M_\sigma$ be a totally ordered subset. We show that $\mathfrak F$ has an upper bound in $\mathfrak M_\sigma$.

Since $X \times X$ is countable, there exists a sequence $(\rhoooo_n)$ in $\mathfrak F$ with
$$ \sup\{\rhoooo_n \mid n \in \N\} = \sup\{\rhoooo \mid \rhoooo \in \mathfrak F\}.$$
Since $\mathfrak F$ is totally ordered, for every $n \in \N$ we have $\rhoooo_n \leq \rhoooo_{n+1}$ or $\rhoooo_n \geq \rhoooo_{n+1}$. We use this to inductively define $r_1  =  \rhoooo_1$ and for $n \geq 2$
$$r_n = \begin{cases}
            \rhoooo_n &\text{if } r_{n-1} < \rhoooo_n\\
            r_{n-1} &\text{if } r_{n-1} \geq \rhoooo_n
           \end{cases}.
$$
Then $r_n \in \mathfrak F$, $r_n \leq r_{n + 1}$  and $\rhoooo_n \leq r_n$. We infer
$$ \sup\{\rhoooo \mid \rhoooo \in \mathfrak F\} = \sup\{\rhoooo_n \mid n \in \N\} = \sup\{r_n \mid n \in \N\} = \lim_{n \to \infty} r_n, $$
where the limit holds pointwise. By
Proposition~\ref{proposition:compactness} pointwise limits of
intrinsic pseudo metrics are again intrinsic pseudo metrics. Hence,
$\sup\{\rhoooo \mid \rhoooo \in \mathfrak F\} \in \mathfrak M_\sigma$ is
an upper bound for $\mathfrak F$.
\end{proof}

The fact that in general $\kappa$ is not an intrinsic metric leads to non-uniqueness of maximal intrinsic metrics.

\begin{theorem}\label{theorem:uniqueness of maximal intrinsic metrics}
 Let $b$ be a locally finite connected graph over $(X,m)$. The following assertions are equivalent.

\begin{enumerate}[(i)]
 \item There exists a unique maximal intrinsic metric.
 \item  $b$ is a star graph and it has a center $p \in X$ such that
 $$\sum_{x \in X \setminus \{p\}} m(x) \leq m(p).$$
\end{enumerate}
\end{theorem}
\begin{proof}
 (ii) $\Rightarrow$ (i): By Theorem~\ref{thm:star graphs} assertion (ii) implies that $\kappa$ is an intrinsic metric. Clearly it is maximal.

 (i) $\Rightarrow$ (ii): Assume (ii) does not hold. According to Proposition~\ref{lemma:finite sups are sufficient} and Theorem~\ref{thm:star graphs} there exist intrinsic pseudo metrics $\rhoooo,\sigma$ such that $\rhoooo \vee \sigma$ is not intrinsic. By Proposition~\ref{prop:existence maximal metrics} we can assume that $\rhoooo$ and $\sigma$ are maximal. Since $\rhoooo \vee \sigma$ is not intrinsic, we infer $\rhoooo \neq \sigma$.
\end{proof}

For certain examples non-uniqueness of maximal intrinsic pseudo
metrics can be seen quite  easily.

\begin{example}
 Consider $X = \mathbb Z$ with $b(n,m) = 1$ if $|n-m| = 1$ and $b(n,m) = 0$, else. Moreover, let $m(n) = 1$, $n \in  \Z$, be the counting measure. For any function $f \colon \Z \to \R$ with
 $$1 = |\nabla f|^2(n) =  (f(n + 1) - f(n))^2 +  (f(n - 1) - f(n))^2 $$
 for  all $n \in \Z$ the pseudo metric $\sigma_f$ is maximal intrinsic. Clearly, there are many such functions on $\mathbb Z$.
\end{example}

\section{Existence of intrinsic pseudo metrics with finite balls}\label{section:existence of metrics with finite balls}

In this section we (partially) characterize when weakly spherically symmetric graphs (see below for a definition) admit intrinsic pseudo metrics with finite balls.  To this end we employ a criterion involving existence of certain cut-off functions.

A weighted graph $b$ over $(X,m)$ is called {\em $\chi$-complete} if there exists a sequence $(\chi_n)$ in $C_c(X)$ with $\chi_n \to 1$ pointwise and $C > 0$ such that
$$|\nabla \chi_n| \leq C $$
for all  $n \in \N$. In this case, we simply call $(\chi_n)$ a {\em sequence of cut-off functions} for $b$ over  $(X,m)$. Let $T \colon \R \to \R, T(x)  = (x \wedge 1)\vee 0$. Then $T$ is $1$-Lipschitz and therefore $|\nabla (T \circ \chi_n)| \leq |\nabla \chi_n| \leq C$. Hence, we may always assume that cut-off functions  satisfy $0 \leq \chi_n \leq 1$. It was recently proven in \cite{LSW21}  that a locally finite graph $b$ over $(X,m)$ is $\chi$-complete if and only if there exists an intrinsic pseudo metric with finite balls.

From now on we assume that the graph $b$ over $(X,m)$ is connected. Recall that $d$ is the combinatorial graph metric. Fix $o \in X$. We say that a function $f \colon X \to \R$ is radially symmetric with respect to $o$  if the value $f(x)$ only depends on $d(x,o)$. In this case, given $x \in X$ with $d(x,o) = r$ we abuse notation and write $f(r)$ for the value $f(x)$. Often $o$ is fixed and we simply say that $f$ is radially symmetric.

A graph $b$ over $(X,m)$ is called  {\em weakly spherically symmetric} if there exists $o \in X$ (the {\em root}) such that the functions $\kappa_{\pm} \colon X \to \R$ defined by
$$ \kappa_{\pm}(x) = \frac{1}{m(x)} \sum_{y \in S_{r\pm 1} } b(x,y), \quad \text{if } x \in S_{r}, $$
are radially symmetric. Here, $S_r = \{x \in X \mid d(x,o) = r\}$ denotes the combinatorial $r$-sphere around $o$ if $r \in \mathbb N_0$¸ and we use the convention $S_{-1} = \emptyset$, which leads to $\kappa_-(o) = 0$. In this case, for radially symmetric  $f \colon X \to \R$ the function $|\nabla  f|^2$ is also radially symmetric and for  $r \geq 1$ it takes the form
\begin{align*}
 |\nabla  f|^2 (r) &= \kappa_-(r) (f(r) - f(r-1))^2 + \kappa_+(r) (f(r) - f(r+1))^2 \\
 &= \frac{1}{m(S_r)}  \kappa_+(r-1)m(S_{r-1}) (f(r) - f(r-1))^2 \\
 & \quad + \frac{1}{m(S_r)}  \kappa_+(r) m(S_{r}) (f(r) - f(r+1))^2 .
\end{align*}
For the second equality we used $ \kappa_+(r) m(S_r) =  \kappa_-(r+1) m(S_{r+1})$,  which follows from
$$\kappa_+(r) m(S_{r}) = \sum_{x \in S_r, y \in S_{r+1}} b(x,y) = \kappa_-(r+1) m(S_{r+1}).$$
This quantity  can be interpreted as the weighted size of the combinatorial boundary of $B_r = \{x \in X \mid d(x,o) \leq r\}$, which is why we write  $|\partial B_r| = \kappa_+(r) m(S_{r})$. Hence, the formula for the gradient reads
$$|\nabla  f|^2 (r)  = \frac{1}{m(S_r)} \left(|\partial B_{r-1}|(f(r) - f(r-1))^2 + |\partial B_r|(f(r) - f(r+1))^2 \right).$$

\begin{theorem}\label{theorem:weakly spherically symmetric}
 Let $b$  be a locally finite and weakly spherically symmetric graph over $(X,m)$ with root $o \in X$. Of the following assertions (i), (ii) and (iii) are equivalent,  (i)/(ii)/(iii) implies (iv) and (iv) implies (v).
 \begin{enumerate}[(i)]

  \item The graph is $\chi$-complete and there exists a sequence of cut-off functions which are radially symmetric.

 \item  There exists an intrinsic pseudo metric  with finite balls $\sigma$ such that the function $\sigma(o,\cdot)$ is radially symmetric.

  \item
  $$\sum_{r = 1}^\infty \frac{\sqrt{m(S_r) \wedge m(S_{r+1})}}{\sqrt{|\partial B_r|}}   = \infty.$$

    \item There exists an intrinsic metric with finite balls.

  \item $$\sum_{r = 1}^\infty \frac{\sqrt{m(S_r)}}{\sqrt{|\partial B_r|}} = \sum_{r = 1}^\infty \frac{1}{\sqrt{\kappa_+(r)}} = \infty.$$
 \end{enumerate}
 If, moreover, there exists $K > 0$ such that $m(S_r) \leq K m(S_{r+1})$ for all $r \geq 1$, then all assertions are equivalent.
 \end{theorem}
\begin{proof}

(i) $\Rightarrow$ (ii): That the existence of a sequence of cut-off functions yields an intrinsic pseudometric with finite balls was shown in \cite[Theorem~A.1]{LSW21}. Here we discuss why the  pseudometric  constructed  in the proof of \cite[Theorem~A.1]{LSW21} is radially symmetric.

Let $(\psi_n)$ be a sequence of cut-off functions which are radially symmetric. As seen in the proof of the implication (iii) $\Rightarrow$ (i)  of \cite[Theorem~A.1]{LSW21}, there is a sequence of cut-off functions  $(\varphi_n)$ with the following properties:
\begin{itemize}
 \item Every  $\varphi_n$ is a  convex combination of finitely many of the functions $\psi_n, n \in \N$.
 \item $ |\nabla \varphi_n| \to 0$ uniformly, as $n \to \infty$.
\end{itemize}
The first property implies that $\varphi_n$ is radially symmetric and the second  yields that we can assume without loss of generality (after possibly passing to a suitable subsequence)
$$\sum_{n = 1}^\infty \av{ |\nabla \varphi_n|}_\infty \leq 1. $$
Moreover, since $\varphi_n \to 1$ pointwise, we can also assume (after passing to another subsequence)
$$\sum_{n = 1}^\infty (1 - \varphi_n(x))^2 < \infty$$
for all $x \in X$.  Under these conditions it is shown in the proof of \cite[Theorem~A.1]{LSW21} that $\sigma \colon X \times X \to [0,\infty)$ given by
$$\sigma(x,y)  = \left(\sum_{n = 1}^\infty (\varphi_n(x) - \varphi_n(y))^2\right)^{1/2}  $$
is an intrinsic metric with finite balls. Clearly, the function $\sigma(o,\cdot)$ is radially symmetric because the functions $\varphi_n, n \in \N,$ are radially symmetric.

(ii) $\Rightarrow$ (i):  The functions
$$\chi_n \colon X \to [0,\infty),\quad  \chi_n(x) = (1 - \sigma(o,x)/n)_+$$
are radially symmetric and satisfy $|\nabla \chi_n| \leq 1/n$ because $\sigma$ is intrinsic,
see e.g. \cite[Proposition~11.29]{KLW21}. Their support is contained in the ball $B^\sigma_n(o)$, which is finite by assumption, and they satisfy $\psi_n \to 1$ pointwise.

(i) $\Rightarrow$ (iii): Let $(\psi_n)$ be a sequence of cut-off functions which are radially symmetric. Then there exists $C \geq 0$ such that for all $r,n \in \N$ we have
$$\kappa_-(r) (\psi_n(r) - \psi_n(r-1))^2 + \kappa_+(r) (\psi_n(r) - \psi_n(r+1))^2 = |\nabla \psi_n|^2(r) \leq C. $$
This implies
$$(\psi_n(r) - \psi_n(r+1))^2 \leq C \left( \frac{1}{\kappa_+(r)} \wedge \frac{1}{\kappa_-(r+1)} \right) = C \frac{ m(S_r) \wedge m(S_{r+1})}{ |\partial B_r|}.$$
For a given $R \geq 1$ we choose $n \in \N$ large enough such that $\psi_n(R) \geq 1/2$. Since $\psi_n$ has finite support, there is $l\geq 1$ such that $\psi_n(r) = 0$ for $r \geq R + l$. These properties of $\psi_n$ yield
 \begin{align*}
  \frac{1}{2} &\leq  |\psi_n(R) - \psi_n(R + l)| \leq \sum_{r = R}^{R + l-1} |\psi_n(r) - \psi_n(r+1)|\\
  &\leq \sqrt{C} \sum_{r = R}^\infty \frac{\sqrt{m(S_r) \wedge m(S_{r+1})}}{\sqrt{|\partial B_r|}}.
 \end{align*}
 Since $R$ was arbitrary, we arrive at (ii).

 (iii) $\Rightarrow$ (i): We consider radially symmetric functions $\chi_n$ defined by $\chi_n(o) = 1$ and, for $r \geq 1$,
 $$\chi_n(r) =  \left(1 - \sum_{k = 0}^{r-1} 1_{[n,\infty)}(k) \frac{\sqrt{m(S_k) \wedge m(S_{k+1})}}{\sqrt{|\partial B_k|}} \right)_+.
$$
Clearly, $\chi_n \to 1$ pointwise and
\begin{align*}
 |\chi_n(r) - \chi_n(r+1)|^2 & \leq \frac{ m(S_r) \wedge m(S_{r+1})}{ |\partial B_r|}.
\end{align*}
The latter inequality and our expression for $|\nabla \chi_n|^2$ in the radially symmetric case yield $|\nabla \chi_n|^2 \leq 2$.   The assumption (ii) implies that $\chi_n$ has compact support and we arrive at (i).

(ii) $\Rightarrow$ (iv): This is trivial

 (iv) $\Rightarrow$ (v):  As discussed in \cite[Theorem~A.1]{LSW21} the existence of an intrinsic pseudo metric with finite balls yields a sequence of cut-off functions  $(\chi_n)$ and we assume $0 \leq \chi_n \leq 1$. We consider radially symmetric functions $\varphi_n$ defined as follows. First we inductively choose a sequence $(x_r)$ with  $x_0 = o$ and   $x_r \in S_r$ such that
 $$|\chi_n(x_r) - \chi_n(x_{r-1})| = \min \{|\chi_n(y) - \chi_n(x_{r-1})| \mid y \sim x_{r-1}, y \in S_r\},  $$
 and  then we set $\varphi_n(r) = \chi_n(x_r)$. Clearly, $\varphi_n$ has finite support and $\varphi_n \to 1$ pointwise.  Moreover, for $r \geq 1$ we obtain by our choice of $x_r$
\begin{align*}
  \kappa_+(r)  (\varphi_n(r) -  \varphi_n(r+1))^2 &= \frac{1}{m(x_r)} \sum_{y \in S_{r+1}}b(x,y) (\chi_n(x_r) -  \chi_n(x_{r+1}))^2 \\
  &\leq \frac{1}{m(x_r)} \sum_{y \in S_{r+1}}b(x,y) (\chi_n(x_r) -  \chi_n(y))^2\\
  & \leq |\nabla \chi_n|^2(x_r) \leq C.
\end{align*}
This implies
 $$(\varphi_n(r) - \varphi_n(r+1))^2 \leq C \frac{1}{\kappa_+(r)} =  C \frac{m(S_r)}{|\partial B_r|}. $$
With this at hand we can argue as in the proof of (i) $\Rightarrow$ (ii) to conclude (iv).

If we assume $m(S_r) \leq Km(S_{r+1})$ for all $r \geq 0$, the  equivalence of  (ii) and (v) is obvious.
\end{proof}

\begin{remark}
%
It remains open whether on any weakly spherically symmetric $\chi$-complete graph the sequence of cut-off functions can be chosen to be radially symmetric. Note that the sequence $(\varphi_n)$ constructed in the proof of (iv) $\Rightarrow$ (v) need not be a sequence of cut-off functions. They satisfy
$$\kappa_+(r) (\varphi_n(r) - \varphi_n(r+1))^2 \leq C$$
but only the weaker bound
\begin{align*}
 \kappa_-(r) (\varphi_n(r) - \varphi_n(r-1))^2 &= \frac{m(S_{r-1})}{m(S_r)} \kappa_+(r-1) (\varphi_n(r) - \varphi_n(r-1))^2 \\
 &\leq C \frac{m(S_{r-1})}{m(S_r)},
\end{align*}
which leads to
$$|\nabla \varphi_n|^2(r) \leq C\left(1 + \frac{m(S_{r-1})}{m(S_r)}\right).$$
Hence, $(\varphi_n)$ is  a sequence of cut-off functions if there exists $K > 0$ such that $m(S_r) \leq Km(S_{r+1})$ for all $r \geq 0$, which is the additional condition in the theorem.
%
\end{remark}

\begin{remark}
 The idea for the proof of implication (iii) $\Rightarrow$ (i) is taken from the proof of \cite[Theorem~3.20]{BGJ19}.
\end{remark}

 Let $D \subseteq [0,\infty)$. For functions $f,g \colon D \to (0,\infty)$ we write $f\sim g$ if there exists $C > 0$ such that $C^{-1} f(t) \leq g(t) \leq C f(t)$ for all sufficiently large $t \in D$. If $m$ satisfies  $m(S_\cdot) \sim f$  for some monotone increasing function $f \colon  \N \to (0,\infty)$, then there exists $K \geq 0$ such that $m(S_r) \leq K m(S_{r+1})$ for all  $r \in \N$.  In this case, all assertions of Theorem~\ref{theorem:weakly spherically symmetric} are equivalent.

\begin{example}[Trees]
Let $m = 1$ be the counting measure on $X$ and let be an infinite radially symmetric connected tree with $b \in \{0,1\}$. Then  $m(S_n) = |S_n|$ and since any $y \in S_{r+1}$ has exactly one neighbor in $S_r$ (this is the defining property of a connected tree), we have $|\partial B_r| = |S_{r+1}|$ and $|S_r| \leq |S_{r+1}|$. For the inequality  we  used that the graph is radially symmetric and infinite.

 Thus, Theorem~\ref{theorem:weakly spherically symmetric} shows that the graph has an intrinsic metric with finite balls if and only if
 $$\sum_{r = 1}^\infty\frac{\sqrt{|S_r|}}{\sqrt{|S_{r+1}|}} = \infty.$$
 This criterion for $\chi$-completeness was obtained in \cite[Proposition~3.16]{BGJ19}. It shows that if $|S_r|$ grows polynomially in $r$, then the graph has an intrinsic pseudo metric with finite balls. If there exists $\alpha > 0$ such that $|S_r| \sim e^{r^\alpha}$, then the graph has an intrinsic pseudo metric with finite balls if and only if $\alpha \leq 1$.
\end{example}

 \begin{example}[Antitrees]\label{exa:antitrees}
 Let $X$ be infinite,  let $m = 1$ be the counting measure and fix $o \in X$. A connected graph $b$ over $(X,m)$ is called {\em radially symmetric antitree} if every vertex from $S_n$ is connected to every vertex from $S_{n + 1}$ and there are no neighbors within $S_n$. We assume that $b$ is such an antitree with $b(x,y) \in \{0,1\}$ for all $x,y \in X$.  Then $m(S_n) = |S_n|$ and since every  $y \in S_{r+1}$ is connected to all $x \in S_r$, we have $|\partial B_r| = |S_r| |S_{r+1}|$. Hence,
 $$\sum_{r = 1}^\infty \frac{1}{\sqrt{|S_r| \vee  |S_{r+1}|}}  = \infty.$$
 implies that the graph has an intrinsic metric with finite balls. For $\chi$-completeness this condition was  observed in \cite[Proposition~3.22]{BGJ19}.  Moreover, the existence of an intrinsic pseudo metric with finite balls yields
 $$\sum_{r = 1}^\infty \frac{1}{\sqrt{|S_r|}}  = \infty.$$
 If there exists $\alpha > 0$ such that $|S_r| \sim r^\alpha$, then the antitree has an intrinsic pseudo metric with finite balls if and only if $\alpha \leq 2$.  Here the 'only if' part seems to be new.
\end{example}

\bibliographystyle{plain}

\bibliography{literatur}

\end{document}